\documentclass{amsart}

\usepackage{amsfonts,amsmath,amssymb,amsthm,mathrsfs,amsxtra}
\usepackage{enumerate,verbatim}

\usepackage[all,2cell,ps]{xy}
\usepackage{hyperref}

\DeclareMathOperator{\Ass}{Ass}

\DeclareMathOperator{\Ext}{Ext}

\DeclareMathOperator{\Hom}{Hom}
\DeclareMathOperator{\Image}{Image}

\DeclareMathOperator{\Tor}{Tor}

\renewcommand{\ge}{\geqslant}
\renewcommand{\le}{\leqslant}

\theoremstyle{plain}
\newtheorem{theorem}{Theorem}
\newtheorem{lemma}[theorem]{Lemma}

\theoremstyle{definition}

\newtheorem{setup}[theorem]{Set-up}
\newtheorem{discussion}[theorem]{Discussion}

\theoremstyle{remark}

\numberwithin{equation}{theorem}

\title[A short proof of a result of Katz and West]{A short proof of a result of Katz and West}

\author[Dipankar Ghosh]{Dipankar Ghosh}
\address{Department of Mathematics, Indian Institute of Technology Bombay, Powai, Mumbai 400076, India}
\email{dipankar@math.iitb.ac.in, dipug23@gmail.com {\rm (Dipankar Ghosh)}}

\author[Tony J. Puthenpurakal]{Tony J. Puthenpurakal}
\email{tputhen@math.iitb.ac.in {\rm (Tony J. Puthenpurakal)}}

\date{February 5, 2016}

\begin{document}

\begin{abstract}
 We give a short proof of a result due to Katz and West: Let $R$ be a Noetherian ring and $I_1,\ldots,I_t$ ideals of $R$.
 Let $M$ and $N$ be finitely generated $R$-modules and $N' \subseteq N$ a submodule. For every fixed $i \ge 0$, the sets
 $\Ass_R\left( \Ext_R^i(M,N/I_1^{n_1}\cdots I_t^{n_t} N') \right)$ and $\Ass_R\left( \Tor_i^R(M,N/I_1^{n_1}\cdots I_t^{n_t} N') \right)$
 are independent of $(n_1,\ldots,n_t)$ for all sufficiently large $n_1,\ldots,n_t$.
\end{abstract}

\maketitle

 Often in mathematics, once an interesting result has been established, other proof of the same result appears. In this article,
 we give a short proof of a result due to Katz and West.
 
 Let $R$ be a commutative Noetherian ring with identity. Let $I$ be an ideal of $R$ and $M$ a finitely generated $R$-module.
 In \cite{Rat76}, Ratliff conjectured about the asymptotic behaviour of the set of associated prime ideals $\Ass_R(R/I^n)$ (when $R$
 is a domain). Subsequently, Brodmann \cite{Bro79} proved that $\Ass_R(M/I^n M)$ stabilizes for $n$ sufficiently large. Thereafter,
 this result was extended to an arbitrary finite collection of ideals by Kingsbury and Sharp in \cite[Theorem~1.5]{KS96}; see also
 \cite[Corollary~1.8(c)]{KMR89}.
 
 In a different direction, Melkersson and Schenzel generalized Brodmann's result by showing that
 $\Ass_R(\Tor_i^R(M, R/I^n))$ is independent of $n$ for all large $n$ and for every fixed $i \ge 0$; see
 \cite[Theorem~1]{MS93}. Recently, in \cite[Corollary~3.5]{KW04}, Katz and West proved all the above results in a more general form:
 
\begin{setup}\label{set up}
 Let $I_1,\ldots,I_t$ be ideals of $R$. Suppose $M$ and $N$ are finitely generated $R$-modules and $N' \subseteq N$ a submodule.
 We set $\mathbb{N} := \{ n \in \mathbb{Z} : n \ge 0 \}$. Fix $i \in \mathbb{N}$. For every $\mathbf{n} := (n_1,\ldots,n_t) \in
 \mathbb{N}^t$, we denote $\mathbf{I}^{\mathbf{n}} := I_1^{n_1}\cdots I_t^{n_t}$ and we set
 \begin{center}
  $W_{\mathbf{n}} := \Ext_R^i\left( M, N/\mathbf{I}^{\mathbf{n}} N' \right) \quad\mbox{and}\quad
  W'_{\mathbf{n}} := \Tor_i^R\left( M, N/\mathbf{I}^{\mathbf{n}} N' \right)$.
 \end{center}
 Let $W := \bigoplus_{\mathbf{n} \in \mathbb{N}^t} W_{\mathbf{n}}$ and $W' := \bigoplus_{\mathbf{n} \in \mathbb{N}^t} W'_{\mathbf{n}}$.
\end{setup}
 
 With this set-up, Katz and West showed that $\Ass_R(W_{\mathbf{n}})$ and $\Ass_R(W'_{\mathbf{n}})$ are
 independent of $\mathbf{n}$ for all $\mathbf{n} \gg 0$. The aim of this article is to give a short 
 proof of this result. {\it We prove all the results here for Ext-modules only. For the analogous result of Tor-modules,
 the proof goes through exactly the same way}.

\begin{discussion}\label{discussion}
 For every $1 \le j \le t$, $\mathbf{e}^j$ denotes the $j$th standard basis element of $\mathbb{N}^t$.
 Let $\mathscr{R}(\mathbf{I}) := \bigoplus_{\mathbf{n} \in \mathbb{N}^t} \mathbf{I}^{\mathbf{n}}$ be the $\mathbb{N}^t$-graded Rees
 ring. The short exact sequence
 \begin{center}
  $0 \to \bigoplus_{\mathbf{n} \in \mathbb{N}^t}
  \mathbf{I}^{\mathbf{n}} N'/\mathbf{I}^{\mathbf{n} + \mathbf{e}^j} N'
  \longrightarrow \bigoplus_{\mathbf{n} \in \mathbb{N}^t} N/\mathbf{I}^{\mathbf{n} + \mathbf{e}^j} N'
  \longrightarrow \bigoplus_{\mathbf{n} \in \mathbb{N}^t} N/\mathbf{I}^{\mathbf{n}} N' \to 0$
 \end{center}
 yields the following exact sequence of $\mathbb{N}^t$-graded $\mathscr{R}(\mathbf{I})$-modules:
 \[
  \bigoplus_{\mathbf{n} \in \mathbb{N}^t}
  \Ext_R^i\left( M,\dfrac{\mathbf{I}^{\mathbf{n}} N'}{\mathbf{I}^{\mathbf{n}+\mathbf{e}^j} N'} \right)
  \stackrel{\Phi_j}{\longrightarrow} W(\mathbf{e}^j) \longrightarrow W \stackrel{\Psi_j}{\longrightarrow}
  \bigoplus_{\mathbf{n} \in \mathbb{N}^t}
  \Ext_R^{i+1}\left( M, \dfrac{\mathbf{I}^{\mathbf{n}} N'}{\mathbf{I}^{\mathbf{n} + \mathbf{e}^j} N'} \right),
 \]
 where $W(\mathbf{e}^j)_{\mathbf{n}} := W_{\mathbf{n} + \mathbf{e}^j}$ for all $\mathbf{n} \in \mathbb{N}^t$.
 Setting $U^j := \Image(\Phi_j)$ and $V^j := \Image(\Psi_j)$, we obtain the following exact sequence of $\mathbb{N}^t$-graded
 $\mathscr{R}(\mathbf{I})$-modules:
 \begin{equation}\label{equation 1}
  0 \longrightarrow U^j \longrightarrow W(\mathbf{e}^j) \longrightarrow W \longrightarrow V^j \longrightarrow 0,
 \end{equation}
 where $U^j$ and $V^j$ are finitely generated $\mathbb{N}^t$-graded $\mathscr{R}(\mathbf{I})$-modules.
\end{discussion}

 We say an $\mathbb{N}^t$-graded module $U$ is {\it eventually zero} (resp. {\it non-zero}) if $U_{\mathbf{n}} = 0$ for all
 $\mathbf{n} \gg 0$ (resp. $U_{\mathbf{n}} \neq 0$ for all $\mathbf{n} \gg 0$). By virtue of \cite[Proposition~5.1]{Wes04}, {\it a
 finitely generated $\mathbb{N}^t$-graded $\mathscr{R}(\mathbf{I})$-module is either eventually zero or eventually non-zero}.

\begin{lemma}\label{lemma: Hilbert func}
 With the {\rm Set-up~\ref{set up}}, if $(R,\mathfrak{m},k)$ is a local ring, then each of $\Hom_R(k,W)$ and
 $\Hom_R(k,W')$ is either eventually zero or eventually non-zero.
\end{lemma}

\begin{proof}
 For every $1 \le j \le t$, in view of \eqref{equation 1}, by setting $X^j := \Image(W(\mathbf{e}^j) \to W)$, we obtain the following
 short exact sequences of $\mathbb{N}^t$-graded $\mathscr{R}(\mathbf{I})$-modules:
 \[
  0 \to U^j \to W(\mathbf{e}^j) \to X^j \to 0 \quad \mbox{and}\quad
  0 \to X^j \to W \to V^j \to 0,
 \]
 which induce the following exact sequences of $\mathbb{N}^t$-graded $\mathscr{R}(\mathbf{I})$-modules:
 \begin{align}
  & 0 \to \Hom_R(k,U^j) \to \Hom_R(k,W(\mathbf{e}^j)) \to \Hom_R(k,X^j) \to Y^j \to 0,\label{equation 2} \\
  & 0 \to \Hom_R(k,X^j) \to \Hom_R(k,W) \to Z^j \to 0,\label{equation 3}
 \end{align}
 where $Y^j$ and $Z^j$ (being submodules of $\Ext_R^1(k,U^j)$ and $\Hom_R(k,V^j)$ respectively) are finitely generated
 $\mathbb{N}^t$-graded $\mathscr{R}(\mathbf{I})$-modules. It can be observed from \eqref{equation 2} and \eqref{equation 3} that if
 any one of $\Hom_R(k,U^j)$, $Y^j$ and $Z^j$ ($1 \le j \le t$) is eventually non-zero, then so is $\Hom_R(k,W)$, and we are done.
 So we may assume that $\Hom_R(k,U^j)$, $Y^j$ and $Z^j$ are eventually zero for all $1 \le j \le t$. In this case, setting
 $f(\mathbf{n}) := \textrm{length}(\Hom_R(k,W_{\mathbf{n}}))$ for all $\mathbf{n} \in \mathbb{N}^t$, in view of the $\mathbf{n}$th
 components of \eqref{equation 2} and \eqref{equation 3}, we obtain that
 $f(\mathbf{n} + \mathbf{e}^j) = f(\mathbf{n})$ for all $1 \le j \le t$ and for all $\mathbf{n} \gg 0$. Therefore $f(\mathbf{n}) = c$
 for all $\mathbf{n} \gg 0$, where $c$ is a constant. The lemma now follows easily.
\end{proof}

 Now we can achieve the aim of this article.
 
\begin{theorem}\label{theorem: stability}
 With the {\rm Set-up~\ref{set up}}, there exists $\mathbf{k} \in \mathbb{N}^t$ such that the sets $\Ass_R(W_{\mathbf{n}})$ and
 $\Ass_R(W'_{\mathbf{n}})$ are independent of $\mathbf{n}$ for all $\mathbf{n} \ge \mathbf{k}$.
\end{theorem}

\begin{proof}
 We first show that $\bigcup_{\mathbf{n} \in \mathbb{N}^t} \Ass_R(W_{\mathbf{n}})$ is finite. For every
 $\mathbf{n} \in \mathbb{N}^t$, the $\mathbf{n}$th component of the exact sequence \eqref{equation 1} (for $j = 1$) gives
 \begin{align*}
  \Ass_R(W_{\mathbf{n} + \mathbf{e}^1}) & \subseteq \Ass_R(U_{\mathbf{n}}^1) \cup \Ass_R(W_{\mathbf{n}}) \\
  & \subseteq \Ass_R(U_{\mathbf{n}}^1) \cup \Ass_R(U_{\mathbf{n} - \mathbf{e}^1}^1) \cup \Ass_A(W_{\mathbf{n}-\mathbf{e}^1})\\
              & ~~\cdots \\
              & \subseteq \Big( \bigcup_{0 \le l \le n_1} \Ass_R\left(U_{(l,n_2,\ldots,n_t)}^1\right) \Big) \bigcup
                \Ass_R\left( W_{(0,n_2,\ldots,n_t)} \right).
 \end{align*}
 Taking union over $\mathbf{n} \in \mathbb{N}^t$, we obtain that
 \[
  \bigcup_{\mathbf{n} \in \mathbb{N}^t} \Ass_R(W_{\mathbf{n}}) \subseteq
  \Big( \bigcup_{\mathbf{n} \in \mathbb{N}^t} \Ass_R(U_{\mathbf{n}}^1) \Big) \bigcup
  \Big( \bigcup_{(n_2,\ldots,n_t) \in \mathbb{N}^{t-1}} \Ass_R\left( W_{(0,n_2,\ldots,n_t)} \right) \Big).
 \]
 Since $U^1$ is finitely generated, the set $\bigcup_{\mathbf{n} \in \mathbb{N}^t} \Ass_R(U_{\mathbf{n}}^1)$ is finite;
 see \cite[Lemma~3.2]{Wes04}. Therefore one obtains that $\bigcup_{\mathbf{n} \in \mathbb{N}^t} \Ass_R(W_{\mathbf{n}})$ is finite
 by using induction on $t$. 
 
 Since $\bigcup_{\mathbf{n} \in \mathbb{N}^t} \Ass_R(W_{\mathbf{n}})$ is finite, it is now enough to prove that for every
 $\mathfrak{p} \in \bigcup_{\mathbf{n} \in \mathbb{N}^t} \Ass_R(W_{\mathbf{n}})$, exactly
 one of the following alternatives must hold: either $\mathfrak{p} \in \Ass_R(W_{\mathbf{n}})$ for all $\mathbf{n} \gg 0$; or
 $\mathfrak{p} \notin \Ass_R(W_{\mathbf{n}})$ for all $\mathbf{n} \gg 0$. Localizing at $\mathfrak{p}$, and replacing
 $R_{\mathfrak{p}}$ by $R$ and $\mathfrak{p}R_{\mathfrak{p}}$ by $\mathfrak{m}$, it is now enough to prove that either
 $\mathfrak{m} \in \Ass_R(W_{\mathbf{n}})$ for all $\mathbf{n} \gg 0$; or $\mathfrak{m} \notin \Ass_R(W_{\mathbf{n}})$ for all
 $\mathbf{n} \gg 0$, which is equivalent to that either $\Hom_R(k, W_{\mathbf{n}}) \neq 0$ for all $\mathbf{n} \gg 0$; or
 $\Hom_R(k, W_{\mathbf{n}}) = 0$ for all $\mathbf{n} \gg 0$, where $k := R/\mathfrak{m}$. The last result follows from
 Lemma~\ref{lemma: Hilbert func}.
\end{proof}

\section*{Acknowledgements}

The first author would like to thank NBHM, DAE, Govt. of India for providing financial support for this study.

\end{document}